\documentclass[12pt, reqno, a4paper]{amsart}

\usepackage{ amssymb, amsmath, enumerate, amsfonts, amsthm, mathrsfs, url, bm}
\usepackage{mathdots}

\usepackage{xcolor}  	
\usepackage[backref]{hyperref}
\hypersetup{
	colorlinks,
	linkcolor={blue!60!black},
	citecolor={blue!60!black},
	urlcolor={red!60!black}
}

\usepackage{color}

\advance \topmargin by -\headheight
\advance \topmargin by -\headsep
\evensidemargin \oddsidemargin
\DeclareRobustCommand{\SkipTocEntry}[5]{}

\numberwithin{equation}{section}

\newtheorem{theorem}{Theorem}[section]
\newtheorem{lemma}[theorem]{Lemma}    
\newtheorem{corollary}[theorem]{Corollary}
\newtheorem{proposition}[theorem]{Proposition}

\newtheorem*{lemma*}{Lemma}

\theoremstyle{remark}

\theoremstyle{definition}
\newtheorem*{definition}{Definition}
\newtheorem*{remark}{Remark}

\theoremstyle{claim}

\newtheorem*{question*}{Question}

\newcommand{\vx}{\mathbf{x}}
\newcommand{\vy}{\mathbf{y}}

\newcommand{\va}{\mathbf{a}}

\newcommand{\rB}{\mathrm{B}}

\newcommand{\rR}{\mathrm{R}}
\newcommand{\rW}{\mathrm{W}}

\newcommand{\cB}{\mathcal{B}}

\newcommand{\N}{\mathbb{N}}
\newcommand{\Z}{\mathbb{Z}}

\newcommand{\Q}{\mathbb{Q}}
\newcommand{\R}{\mathbb{R}}
\newcommand{\C}{\mathbb{C}}

\newcommand{\E}{\mathbb{E}}

\DeclareMathOperator*{\BigE}{\mathbb{E}}

\newcommand{\eps}{\varepsilon}

\newcommand{\ldn}{\underline{d}}
\newcommand{\udn}{\bar{d}}

\newcommand{\spt}{\mathcal{E}}

\newcommand{\Tow}{\mathrm{tow}}

\makeatletter
\let\@@pmod\pmod
\DeclareRobustCommand{\pmod}{\@ifstar\@pmods\@@pmod}
\def\@pmods#1{\mkern4mu({\operator@font mod}\mkern 6mu#1)}
\makeatother

\begin{document}

	\title[Partition Regularity and Multiplicatively Syndetic Sets]{Partition Regularity and Multiplicatively Syndetic Sets\\}
	\author{Jonathan Chapman}
	\address{Department of Mathematics\\ University of Manchester\\ Oxford Road\\ Manchester\\ M13 9PL\\ UK}
	\email{jonathan.chapman@manchester.ac.uk}
	\date{\today}
	
\begin{abstract}
	We show how multiplicatively syndetic sets can be used in the study of partition regularity of dilation invariant systems of polynomial equations. In particular, we prove that a dilation invariant system of polynomial equations is partition regular if and only if it has a solution inside every multiplicatively syndetic set. We also adapt the methods of Green-Tao and Chow-Lindqvist-Prendiville to develop a syndetic version of Roth's density increment strategy. This argument is then used to obtain bounds on the Rado numbers of configurations of the form $\{x,d,x+d,x+2d\}$.
\end{abstract}

\maketitle

\setcounter{tocdepth}{1}
%toc depth controls what appears on toc;
%0: Chapters only;
%1: Chapters and sections; etc.
\tableofcontents

\section{Introduction}\label{secIntro}

A system of equations is called \emph{partition regular} if, in any finite colouring of the positive integers $\N=C_{1}\cup\cdots\cup C_{r}$, there exists a non-trivial monochromatic solution $\vx=(x_{1},\dots,x_{s})$, meaning that $\vx\in C_{k}^{s}$ for some $k$, and $x_{i}\neq x_{j}$ for some $i\neq j$. The foundational results in the study of partition regularity are the theorems of Schur \cite{Sch16} and van der Waerden \cite{Wae27}. Schur's theorem states that the equation $x+y=z$ is partition regular, whilst van der Waerden's theorem shows that any finite colouring of $\N$ yields arbitrarily long monochromatic (non-trivial) arithmetic progressions.

The theorems of Schur and van der Waerden are both examples of partition regularity being exhibited by certain linear systems of equations. In particular, these systems are \emph{dilation invariant}, meaning that if $\vx=(x_{1},\dots,x_{s})$ is a solution, then so is $\lambda\vx=(\lambda x_{1},\dots,\lambda x_{s})$ for any $\lambda\in\Q$. In this paper we study the properties of general dilation invariant systems of equations, not just those which are linear. We show that the regularity of such systems is inexorably connected with a special class of sets known as \emph{multiplicatively syndetic sets}.

\subsection{Syndeticity}

Syndetic sets originate from the study of topological dynamics of semigroups (see \cite{SyndNotes,HS12}). Given a semigroup $(G,\cdot)$, a set $S\subseteq G$ is called \emph{(left)-syndetic} if there exists a finite set $F\subseteq G$ such that, for each $g\in G$, we have $S\cap(g\cdot F)\neq\emptyset$. Here $g\cdot F:=\{gt:t\in F\}$.

The most familiar notion of syndeticity arises in the additive setting where $(G,\cdot)=(\N,+)$. In this case a syndetic subset $S$ is called \emph{additively syndetic} and is just an infinite set with `bounded gaps'. That is, $S$ is additively syndetic if and only if $S=\{a_{1},a_{2},\dots \}$ for some infinite sequence $a_{1}<a_{2}<\dots$ such that the gaps $|a_{n+1}-a_{n}|$ are uniformly bounded.

In this paper, we study syndetic sets in the multiplicative semigroup $(\N,\cdot)$.

\begin{definition}[Multiplicatively syndetic set]
	Let $F\subset\N$ be a non-empty finite set. We say that $S\subseteq\N$ is a \emph{multiplicatively $F$-syndetic set} if, for every $a\in\N$, we have $S\cap(a\cdot F)\neq\emptyset$.
\end{definition}
\noindent Multiplicatively syndetic sets possess a number of interesting properties. Graham, Spencer, and Witsenhausen \cite{GrahamSyndy} observed that multiplicatively syndetic sets have positive density. Much later, Bergelson \cite[Lemma 5.11]{BerCent} used methods from ultrafilter theory to show that multiplicatively syndetic sets are \emph{additively central}\footnote{A subset of $\N$ is called \emph{additively central} if it is a member of
a minimal idempotent ultrafilter on $(\N,+)$ (see \cite[Definition 5.8]{BerCent}).} (which implies that they have positive density).

The fact that multiplicatively syndetic sets have positive density plays a significant role in the work of Chow, Lindqvist and Prendiville \cite{CLP}. They demonstrate how multiplicatively syndetic sets can be used to obtain partition regularity results for non-linear equations via an ``induction on colours'' argument. Their work shows that a sufficient condition for a dilation invariant equation to be partition regular is that it has a solution inside\footnote{A system $\spt$ is said to have a solution in a set $S$ if there exists a solution $\vx$ to $\spt$ with each entry of $\vx$ lying in $S$.} every multiplicatively syndetic set. Our first main theorem is a converse of this result.

\begin{theorem}[Partition regularity is equivalent to syndetic solubility]\label{thmHomPR}
	Let $\spt$ be a dilation invariant finite system of equations. Then $\spt$ is partition regular if and only if $\spt$ has a non-trivial solution inside every multiplicatively syndetic set.
\end{theorem}

\noindent As an immediate corollary to this theorem, we obtain the following \emph{dilation invariant consistency theorem}.

\begin{corollary}[Dilation invariant consistency theorem]\label{corConsis}
	Let $\spt_{1},\dots,\spt_{s}$ be $s$ dilation invariant partition regular finite systems of equations. Then in any finite colouring $\N=C_{1}\cup\cdots\cup C_{r}$ there exists a colour class $C_{t}$ such that each $\spt_{i}$ has a solution inside $C_{t}$.
\end{corollary}

\subsection{Brauer Configurations}

Van der Waerden \cite{Wae27} proved that, for all $r,k\in\N$, there exists a (minimal) positive integer $W(r,k)\in\N$ such that, in any $r$-colouring of the set $\{1,\dots,W(r,k)\}$, there exists a monochromatic arithmetic progression of length $k$. Obtaining good bounds for $W(r,k)$ is a notoriously difficult problem. Over 60 years after van der Waerden's original paper, Shelah \cite{Shelah} obtained the first primitive recursive bounds. The best bounds currently known are due to Gowers \cite{Gow01} who obtained the bound
\begin{equation}\label{GowBd}
\rW(r,k)\leqslant 2^{2^{r^{2^{2^{k+9}}}}}.
\end{equation}

In \S\ref{secBrau} we consider a variation of van der Waerden's theorem concerning configurations of the form 
\begin{equation*}
\{x,d,x+d,x+2d\}.
\end{equation*}
These are arithmetic progressions of length 3 along with their common difference. Brauer \cite{Bra28} was the first to establish the partition regularity of these configurations, and so we refer to them as \emph{Brauer configurations (of length 3)}. We also call the corresponding Rado numbers the \emph{($r$-colour) Brauer numbers}. Specifically, we define $\rB(r)\in\N$ to the the smallest positive integer such that every $r$-colouring of the interval $\{1,\dots,\rB(r)\}$ yields a monochromatic set of the form $\{x,d,x+d,x+2d\}$. 

To show that Brauer configurations (of length 3) are partition regular, Theorem \ref{thmHomPR} informs us that it is sufficient to prove that all multiplicatively syndetic sets contain such configurations. Our next result establishes a quantitative version of Brauer's theorem for multiplicatively syndetic sets.

\begin{theorem}\label{thmSynBrau}
	There exists a positive absolute constant $c>0$ such that the following is true. Let $S\subseteq\N$ be a multiplicatively $F$-syndetic set, for some non-empty finite set $F\subset\N$. Let $M$ denote the largest element of $F$. If $N\geqslant 3$ satisfies
	\begin{equation*}
	M\leqslant\exp\left(c\sqrt{\log\log N}\right), 
	\end{equation*}
	then there exists $d,x\in\N$ such that $\{x,d,x+d,x+2d\}\subseteq S\cap\{1,2,\dots,N\}.$
\end{theorem}

This theorem is analogous to Green and Tao's result \cite[Theorem 1.1]{GT09} that sets $A\subseteq\{1,2,\dots,N\}$ which lack $4$-term arithmetic progressions have size
\begin{equation*}
|A|\leqslant N\exp\left(-c\sqrt{\log\log N}\right).
\end{equation*} 
We deduce Theorem \ref{thmSynBrau} from a more general density result (Theorem \ref{thmSynDensInc}), which concerns dense sets $A\subseteq\{1,2,\dots,N\}$ lacking arithmetic progressions of length $3$ with common difference lying in a given multiplicatively syndetic set $S$. This density result is proven in \S\ref{secBrau} by combining the methods of Green and Tao \cite{GT09} with a `multiplicatively syndetic induction on colours' argument of Chow, Lindqvist, and Prendiville \cite{CLP}.

Brauer's theorem may be proved by iteratively applying van der Waerden's theorem. As indicated by Cwalina and Schoen \cite{Cwalina}, the best bounds one can obtain for the Brauer numbers by incorporating Gowers' bound (\ref{GowBd}) into this argument are of the form
\begin{equation*}
\rB(r)\leqslant\Tow\left((5+o(1))r\right).
\end{equation*}
Here $\Tow(n)$ denotes an exponential tower of $2$'s of height $n$. Explicitly, we take $\Tow(1):=2$ and for all $n\geqslant 2$ define
\begin{align*}
\Tow(n):=2^{\Tow(n-1)}.
\end{align*}
By incorporating Theorem \ref{thmSynBrau} into an induction on colours argument, we obtain an asymptotic improvement on this bound.

\begin{theorem}[Tower bound for $\rB(r)$]\label{thmBrauerBound}
	For each $r\in\N$,
	\begin{equation}
	\rB(r)\leqslant \Tow\left((1+o(1))r\right).
	\end{equation}
\end{theorem}

In general, for a partition regular system of equations $\spt$, one can define the \emph{$r$-colour Rado number} $\rR_{\spt}(r)$ to be the smallest $N\in\N$ such that every $r$-colouring of the interval $\{1,\dots,N\}$ yields a monochromatic solution to $\spt$. Cwalina and Schoen \cite[Theorem 1.5]{Cwalina} proved that if $\spt$ is a partition regular homogeneous linear equation of the form
\begin{equation*}
a_{1}x_{1}+\cdots+a_{s}x_{s}=0,
\end{equation*}
where $a_{1},\dots,a_{s}\in\Z\setminus\{0\}$, then
\begin{equation*}
\rR_{\spt}(r)\ll_{\spt} 2^{O_{\spt}(r^{4}\log r)}.
\end{equation*} 
The improvements obtained by Cwalina and Schoen for single equations ultimately derive from the fact that single linear equations are controlled by the $U^{2}$ norm (see \S\ref{secBrau} for a definition of the $U^{s}$ norms), and so they can be analysed with (linear) Fourier analysis. However, Brauer configurations of length $3$ are controlled by the $U^{3}$ norm and therefore require methods from quadratic Fourier analysis. In \cite{ChapPren} we use higher order Fourier analysis to improve on Theorem \ref{thmBrauerBound} by obtaining a double exponential bound of the form
\begin{equation*}
\rB(r)\leqslant\exp\exp(r^{C}).
\end{equation*}
More generally, we show that a bound of the above form holds for Brauer configurations of any length $k$ (with constant $C$ depending on the length $k$).
\subsection*{Notation}

The positive integers are denoted by $\N$. Given $X\geqslant 1$, we let $[X]:=\{n\in\N: 1\leqslant n\leqslant X\}=\{1,2,\dots,\lfloor X\rfloor \}$.

Let $f$ and $g$ be positively valued functions. We write $f\ll g$, or $g\gg f$, or $f=O(g)$ if there exists a positive constant $C$ such that $f(x)\leqslant C g(x)$ for all $x$. If we require the constant $C$ to depend on some parameters $\lambda_{1},\dots,\lambda_{k}$, then we write $f\ll_{\lambda_{1},\dots,\lambda_{k}} g$ or $f=O_{\lambda_{1},\dots,\lambda_{k}}(g)$. 

The letters $c$ and $C$ are typically used to denote absolute constants, whose values may change from line to line. We usually write $c$ to denote a small constant $0<c<1$, whereas $C$ usually denotes a large constant $C>1$.

\subsection*{Acknowledgements}

The author would like to thank Sean Prendiville for his constant support and encouragement, and for his helpful comments on an earlier draft of this paper. We also thank the anonymous referee for their comments and suggestions on an earlier version of this paper.

\section{Multiplicative Syndeticity and Partition Regularity}\label{secDefProp}

We begin by formally introducing the concepts of partition regularity and multiplicative syndeticity mentioned in the introduction. After establishing the basic properties of multiplicatively syndetic sets, we prove Theorem \ref{thmHomPR} and Corollary \ref{corConsis}.

\subsection{Systems of Equations}

 We consider finite systems of polynomial equations $\spt$ in $s\in\N$ variables of the form
\begin{align}\label{defSysEq}
p_{1}(t_{1},t_{2},\dots,t_{s})&=0;\nonumber\\
p_{2}(t_{1},t_{2},\dots,t_{s})&=0;\nonumber\\
&\vdots\\
p_{k}(t_{1},t_{2},\dots,t_{s})&=0,\nonumber
\end{align}
where each $p_{i}\in\Q[t_{1},t_{2},\dots,t_{s}]$ is a polynomial in the variables $\{t_{i}\}_{i=1}^{s}$.
In this paper we only consider systems of finitely many equations, each with finitely many variables. For related results concerning the regularity of infinite systems, see \cite{BHLS15,HLS03}.

We usually refer to such a system of polynomial equations $\spt$ simply as a \emph{system of equations}. We call $\vx\in\N^{s}$ a \emph{solution} to the system $\spt$ if $p_{i}(\vx)=0$ for all $i$, meaning that $\vx$ is a solution to all of the equations in $\spt$ simultaneously. A solution $\vx=(x_{1},\dots,x_{s})$ is called a \emph{non-trivial solution} if the entries of $\vx$ are not all equal, meaning that $x_{i}\neq x_{j}$ for some $i\neq j$. Given a set $S\subseteq\Q$, we say that $\spt$ has a (non-trivial) solution in $S$ if there is a (non-trivial) solution $\vx=(x_{1},\dots,x_{s})\in\N^{s}$ to $\spt$ such that $x_{i}\in S$ for all $i$.

A system of equations $\spt$ is called \emph{dilation invariant} if the following is true. If $\vx=(x_{1},\dots,x_{s})$ is a solution to $\spt$, then $\lambda\vx=(\lambda x_{1},\dots,\lambda x_{s})$ is also a solution for every $\lambda\in\Q$. For the majority of this paper, we restrict our attention to dilation invariant systems of polynomial equations. However it should be noted that most of the results we prove in this section apply to any dilation invariant system of equations and not just those consisting of polynomial equations.

\subsection{Partition Regularity}

As mentioned in the introduction, the partition regularity of equations is a well-studied topic in Ramsey theory. Recall that an \emph{$r$-colouring} of a set $X$ is a partition $X=C_{1}\cup\cdots\cup C_{r}$ of $X$ into $r$ \emph{colour classes} $C_{i}$. Equivalently, an $r$-colouring can be defined by a function $\chi:X\to A$, for some set $A=\{a_{1},\dots,a_{r}\}$ with $|A|=r$ (usually we take $A=[r]$). These two characterisations can be seen to be equivalent by taking $\chi^{-1}(a_{i})=C_{i}$. A subset $Y\subseteq X$ is called \emph{$(\chi)$-monochromatic} if $\chi$ is constant on $Y$, or equivalently that $Y\subseteq C_{i}$ for some colour class $C_{i}$.

\begin{definition}[Partition regularity]
	Let $S\subseteq\Q$ be a non-empty set and let $\spt$ be a system of equations with coefficients in $\Q$. Let $r\in\N$. We say that $\spt$ is \emph{(kernel) $r$-regular over $S$} if, for each $r$-colouring $\chi:\Q\to[r]$, there exists a $\chi$-monochromatic non-trivial solution $\vx$ to $\spt$ with entries in $S$. We call such an $\vx$ a ($\chi$-)\emph{monochromatic (non-trivial) solution} to $A$. We say that $\spt$ is \emph{(kernel) partition regular over $S$} if $\spt$ is $r$-regular over $S$ for every $r\in\N$. 
\end{definition}

In practice, when one shows that a given system of equations $\spt$ is $r$-regular, the proof actually yields a number $\rR_{\spt}(r)$ (known as the \emph{$r$ colour Rado number} for $\spt$) such that $\spt$ is $r$-regular over the finite interval $[\rR_{\spt}(r)]$. This is certainly the case whenever one obtains a quantative regularity result, such as in \cite{Cwalina,Gow01,Sch16,Wae27}. By assuming (some form of) the axiom of choice, one can show that if $\spt$ is $r$-regular, then such an $\rR_{\spt}(r)$ necessarily exists. This result is known as the compactness principle.\vspace{3mm}

\noindent\textbf{Compactness Principle.} \emph{Let $\spt$ be a finite system of equations in finitely many variables. Let $A\subseteq\N$ and let $r\in\N$. Then $\spt$ is $r$-regular over $A$ if and only if there exists a finite set $F\subseteq A$ such that $\spt$ is $r$-regular over $F$.}
\begin{proof}
	See \cite[Theorem 4]{GRS90}.
\end{proof}

\begin{remark}
	For the rest of this section, we assume (some form of) the axiom of choice in order to make use of the compactness principle. This assumption is not required for any of the remaining sections.
\end{remark}

\subsection{Multiplicatively Thick Sets}
The compactness principle informs us that a system of equations $\spt$ is partition regular if and only if, for each $r\in\N$, we can find a finite set $F_{r}\subset N$ such that $\spt$ is $r$-regular over $F_{r}$. Thus, a sufficient condition for $\spt$ to be partition regular over a set $A\subseteq\N$ would be that $F_{r}\subseteq A$ for all $r\in\N$. The problem with this condition is that it is quite possible that the only set which could satisfy this property is $A=\N$. If $\spt$ is a dilation invariant system of equations, then we can relax this condition to the requirement that, for each $r\in\N$, we can find $t_{r}\in\N$ such that $t_{r}\cdot F_{r}\subseteq A$. This motivates the following definition.

\begin{definition}[Multiplicatively thick set]
	Let $T\subseteq\N$. We say that $T$ is a \emph{multiplicatively thick set} if, for each finite set $F\subset\N$, there exists $t\in\N$ such that $t\cdot F\subseteq T$.
\end{definition}

\begin{proposition}[Regularity over thick sets]\label{propThickEqv}
	Let $\spt$ be a dilation invariant system of equations. Let $r\in\N$. Then the following are all equivalent:
	\begin{enumerate}[\upshape(I)]
		\item $\spt$ is $r$-regular;
		\item $\spt$ is $r$-regular over every multiplicatively thick set;
		\item $\spt$ is $r$-regular over some multiplicatively thick set $T$.
	\end{enumerate}
\end{proposition}
\begin{proof}
	The implications (II)$\Rightarrow$(III) and (III)$\Rightarrow$(I) are immediate. It only remains to establish (I)$\Rightarrow$(II).
	
	Suppose $\spt$ is $r$-regular. By compactness, we can find a finite set $F\subset\N$ such that $\spt$ is $r$-regular over $F$. Now let $T\subseteq\N$ be a multiplicatively thick set. We can then find $t\in\N$ such that $t\cdot F\subseteq T$. Now suppose $\chi:T\to [r]$ is an $r$-colouring of $T$. Define a new $r$-colouring $\tilde{\chi}:F\to [r]$ of $F$ by $\tilde{\chi}(x)=\chi(tx)$. Since $\spt$ is $r$-regular over $F$, we can find a $\tilde{\chi}$-monochromatic solution $\vx$ to $\spt$ in $F$. By dilation invariance, we deduce that $t\vx$ is a $\chi$-monochromatic solution to $\spt$ in $T$. Thus $\spt$ is $r$-regular over $T$.
\end{proof}

\subsection{Multiplicatively Syndetic Sets}
We have now reduced regularity over $\N$ to regularity over a multiplicatively thick set. The utility of Proposition \ref{propThickEqv} is demonstrated in the following argument. Suppose that we have a dilation invariant system of equations $\spt$ and an integer $r>1$ such that $\spt$ is $(r-1)$-regular. We would like to use this to test whether $\spt$ is $r$-regular. Suppose that we have an $r$-colouring $\N=C_{1}\cup\cdots\cup C_{r}$. Informally, if we know that one of the colour classes $C_{j}$ is `small', then we would expect, by $(r-1)$-regularity, to find a monochromatic solution in a colour class $C_{i}$ with $i\neq j$. 

To make this rigorous, suppose that we have a colour class $C_{j}$ such that the complement $\N\setminus C_{j}$ is multiplicatively thick. The remaining $(r-1)$ colour classes induce an $(r-1)$-colouring on $\N\setminus C_{j}$. By Proposition \ref{propThickEqv}, since $\spt$ is $(r-1)$-regular, we can find a monochromatic solution to $\spt$ inside $\N\setminus C_{j}$.

This shows that if the dilation invariant system $\spt$ is $(r-1)$-regular but not $r$-regular, then there is an $r$-colouring $\N=C_{1}\cup\cdots C_{r}$ without non-trivial monochromatic solutions to $\spt$ such that each complement $\N\setminus C_{i}$ is not multiplicatively thick. Observe that $\N\setminus C_{i}$ is not multiplicatively thick if and only if there exists a finite set $F\subset\N$ such that, for every $n\in\N$, we have $(n\cdot F)\cap C_{i}\neq\emptyset$. This motivates the following definition.

\begin{definition}[Multiplicatively syndetic set]
	Let $S\subseteq\N$. Let $F\subset\N$ be a non-empty finite set. We say that $S$ is \emph{multiplicatively $F$-syndetic} if, for each $n\in\N$, we can find some $t\in F$ such that $nt\in S$. Equivalently, for every $n\in\N$, we have $(n\cdot F)\cap S\neq\emptyset$.
	We call $S\subseteq\N$ \emph{multiplicatively syndetic} if $S$ is multiplicatively $F$-syndetic for some non-empty finite set $F\subset\N$.
\end{definition}

\begin{remark}
Chow, Lindqvist, and Prendiville \cite{CLP} define an \emph{$M$-homogeneous set} to be a set which intersects every homogeneous arithmetic progression $x\cdot[M]$ of length $M$ for every $x\in\N$. We therefore observe that an $M$-homogeneous set is exactly the same as a multiplicatively $[M]$-syndetic set.
\end{remark}

As mentioned previously, multiplicatively syndetic sets can be equivalently defined in terms of multiplicatively thick sets.

\begin{proposition}
	Let $S\subseteq \N$. Then the following are all equivalent:
	\begin{enumerate}[\upshape(I)]
		\item $S$ is multiplicatively syndetic;
		\item for every multiplicatively thick set $T$, we have $S\cap T\neq\emptyset$;
		\item $\N\setminus S$ is not multiplicatively thick.
	\end{enumerate}
\end{proposition}
\begin{proof}$ $
	
	(I)$\Rightarrow$(II): Suppose $S$ is multiplicatively $F$-syndetic for some $F\subset\N$, and suppose $T\subseteq\N$ is a multiplicatively thick set. This means that we can find $t_{T}\in\N$ such that $t_{T}\cdot F\subseteq T$. Since $S$ is multiplicatively $F$-syndetic, we have $(t_{T}\cdot F)\cap S\neq\emptyset$. In particular, $S\cap T\neq\emptyset$.
	
	(II)$\Rightarrow$(III): Follows from the fact that $S$ and $\N\setminus S$ are disjoint.
	
	(III)$\Rightarrow$(I): Since $\N\setminus S$ is not multiplicatively thick, we can find a non-empty finite set $F\subset\N$ such that $t\cdot F\nsubseteq \N\setminus S$ for every $t\in\N$. This implies that $S$ is multiplicatively $F$-syndetic.
\end{proof}

In Proposition \ref{propThickEqv} we showed that a dilation invariant system of equations is $r$-regular if and only if it is $r$-regular over all multiplicatively thick sets. This is a consequence of the `largeness' of multiplicatively thick sets. We now prove a similar result for multiplicatively syndetic sets. To do this, we identify multiplicatively syndetic sets with finite colourings in the following manner.

\begin{definition}[Encoding function]
	Let $S\subseteq\N$ be a multiplicatively $F$-syndetic set, for some non-empty finite  $F\subset\N$. The \emph{encoding function for $(S,F)$} is the function $\tau_{S;F}:\N\to F$ defined by
	\begin{equation*}
	\tau_{S;F}(n):=\min\{t\in F:nt\in S \}.
	\end{equation*}
\end{definition}
\noindent Note that the assertion that $\tau_{S;F}$ is a well-defined total function is equivalent to the statement that $S$ is multiplicatively $F$-syndetic.

The encoding function $\tau_{S;F}$ defines a finite colouring of $\N$. Moreover, if a set $A$ is monochromatic with respect to this colouring, then there exists some $t\in F$ such that $t\cdot A\subseteq S$. This observation leads to the following result.

\begin{proposition}[Syndetic sets contain PR configurations]\label{propSSPR}
	Let $A\subseteq\N$. Let $S\subseteq\N$ be a multiplicatively $F$-syndetic set, for some non-empty finite set $F\subset\N$. Let $\spt$ be a dilation invariant system of equations, and let $r\in\N$. If $\spt$ is $(|F|\cdot r)$-regular over $A$, then $\spt$ is $r$-regular over $S\cap(F\cdot A)$.
\end{proposition}
\begin{proof}
	Suppose $\chi:S\cap(F\cdot A)\to[r]$ is an $r$-colouring. Let $\tau=\tau_{S;F}$. Now let $\tilde{\chi}:A\to F\times [r]$ be the product colouring given by
	\begin{equation*}
	\tilde{\chi}(n):=(\tau(n),\chi(n\tau(n))).
	\end{equation*}
	Since $\spt$ is $(|F|\cdot r)$-regular over $A$, we can find a $\tilde{\chi}$-monochromatic solution $\va$ to $\spt$ whose entries $a_{i}$ all lie in $A$. From the definition of $\tilde{\chi}$, we can find $t\in F$ such that $\tau(a_{i})=t$ for each entry $a_{i}$. From the dilation invariance of $\spt$, we deduce that $\vx:=t\va$ is a $\chi$-monochromatic solution to $\spt$ whose entries $x_{i}=ta_{i}$ all lie in $S\cap (F\cdot A)$.
\end{proof}

\noindent This proposition immediately gives the following corollary.

\begin{corollary}\label{corPRsyn}
	Let $\spt$ be a dilation invariant system of equations. Then the following are all equivalent:
	\begin{enumerate}[\upshape(I)]
		\item $\spt$ is partition regular (over $\N$);
		\item $\spt$ is partition regular over every multiplicatively syndetic set;
		\item $\spt$ is partition regular over some multiplicatively syndetic set $S\subseteq\N$.
	\end{enumerate}
\end{corollary}

We have thus shown that partition regularity over $\N$ is equivalent to partition regularity over a particular multiplicatively syndetic set. Our goal now is to prove Theorem \ref{thmHomPR} and therefore show that partition regularity over $\N$ is actually equivalent to $1$-regularity over every multiplicatively syndetic set. 

Recall that our motivation for introducing multiplicatively syndetic sets came from considering colourings in which some of the colour classes were not multiplicatively thick. This leads to the following induction argument first developed in \cite{CLP} to establish partition regularity of certain non-linear dilation invariant equations.

\begin{lemma}[Induction on colours schema]\label{lemColInd}
	Let $\spt$ be a dilation invariant system of equations. If $\spt$ is $r$-regular (for some $r\in\N$), then there exists a finite set $F=F(\spt,r)\subset\N$ so that the following holds. If $\N=C_{1}\cup\cdots\cup C_{r+1}$ is an $(r+1)$-colouring which lacks monochromatic solutions to $\spt$, then each colour class $C_{i}$ must be a multiplicatively $F$-syndetic set.
\end{lemma}
\begin{proof}
	Since $\spt$ is $r$-regular, the compactness principle allows us to find a non-empty finite set $F=F(\spt,r)\subset\N$ such that $\spt$ is $r$-regular over $F$. By dilation invariance, in any colouring $\chi$ of $\N$, if there exists a set of the form $x\cdot F$ (with $x\in\N$) which receives at most $r$ distinct colours, then there exists a $\chi$-monochromatic solution to $\spt$ in $x\cdot F$. By contraposition we deduce that if $\N=C_{1}\cup\cdots\cup C_{r+1}$ is an $(r+1)$-colouring which lacks monochromatic solutions to $\spt$, then each colour class $C_{i}$ is a multiplicatively $F$-syndetic set.
\end{proof}

This lemma shows that when we are trying to prove that a given dilation invariant system $\spt$ is partition regular, we only need to consider colourings in which all of the colour classes are multiplicatively syndetic. Combining this with Corollary \ref{corPRsyn} allows us to prove Theorem \ref{thmHomPR}.

\begin{proof}[Proof of Theorem \ref{thmHomPR}]
	If $\spt$ is partition regular, then Corollary \ref{corPRsyn} implies that $\spt$ is partition regular over every multiplicatively syndetic set. In particular, $\spt$ has a non-trivial solution inside every multiplicatively syndetic set.
	
	Conversely, suppose $\spt$ is not partition regular. If $\spt$ is not $1$-regular, then $\spt$ has no non-trivial solutions in the multiplicatively syndetic set $\N$. Suppose then that $\spt$ is $1$-regular. By Lemma \ref{lemColInd}, there exists a finite colouring of $\N$ with no monochromatic non-trivial solutions to $\spt$ and with each colour class being a multiplicatively syndetic set. Therefore each colour class is a multiplicatively syndetic set which has no non-trivial solutions to $\spt$. 
\end{proof}

This result therefore reduces the task of establishing $r$-regularity over $\N$ for every $r\in\N$ to establishing solubility in every multiplicatively syndetic set. Whilst this may not immediately appear to be helpful, we can obtain Corollary \ref{corConsis} very easily from this new approach.

\begin{proof}[Proof of Corollary \ref{corConsis}]
	For each $k\in[s]$, let $m_{k}$ denote the number of variables appearing in the equations defining the system $\spt_{k}$. We can therefore define a dilation invariant system $\spt$ in $m=m_{1}+\cdots+m_{s}$ variables whose solutions are precisely tuples of the form $(\vx^{(1)},\dots,\vx^{(s)})$, where $\vx^{(k)}\in\Q^{m_{k}}$ is a solution to the system $\spt_{k}$. Since each $\spt_{i}$ is partition regular, it follows from Theorem \ref{thmHomPR} that $\spt$ is partition regular. This implies the desired result.
\end{proof}

\begin{remark}
	In the case that each $\spt_{i}$ is a partition regular linear homogeneous equation, the above result is an immediate consequence of Rado's Criterion \cite[Satz IV]{Rado}. Our proof shows that it is not necessary to utilise such a strong result.
\end{remark}

\subsection{Multiplicatively Piecewise Syndetic Sets}

By using encoding functions, one can show that for a non-empty finite set $F\subset\N$, a set $S\subseteq\N$ is multiplicatively $F$-syndetic if and only if
\begin{equation*}
\N=\bigcup_{t\in F}t^{-1}S,
\end{equation*}
where $t^{-1}S:=\{n\in\N:nt\in S \}$. Our proof of Proposition \ref{propSSPR} used this fact to show that regularity over $\N$ can be `lifted' to regularity over a multiplicatively syndetic set. However, we proved in Proposition \ref{propThickEqv} that regularity over $\N$ is equivalent to regularity over a multiplicatively thick set. This motivates the introduction of the following weaker form of syndeticity.

\begin{definition}[Multiplicatively piecewise syndetic set]
Let $F\subset\N$ be a non-empty finite set, and let $S\subseteq\N$. We say that $S$ is \emph{multiplicatively piecewise $F$-syndetic} if the set $\cup_{t\in F}\,(t^{-1}S)$ is a multiplicatively thick set. \newline
We call $S\subseteq\N$ \emph{(multiplicatively) piecewise syndetic} if $S$ is multiplicatively piecewise $F$-syndetic for some $F\subset\N$.
\end{definition}
Another way to view multiplicatively piecewise syndetic sets is through the following `partial encoding' formulation. Given a non-empty finite set $F\subset\N$ and a set $S\subseteq\N$, define a partial function\footnote{We use the partial function notation $f:A\nrightarrow B$ to mean that $f$ only defines a function on a (possibly empty) subset of $A$.} $\tau_{S;F}:\N\nrightarrow F$ by
\begin{equation}\label{eqnTauDef}
\tau_{S;F}(n):=\min\{t\in F:nt\in S \},
\end{equation}
for all $n\in\N$ for which the above quantity is defined. We refer to this partial function as the \emph{(partial) encoding function for $(S,F)$}. By the \emph{domain of $\tau_{S;F}$} we mean the set of all $n\in\N$ for which (\ref{eqnTauDef}) is defined. 

We remarked earlier that $S$ is multiplicatively $F$-syndetic if and only if $\tau_{S;F}$ is a total function, meaning that $\tau_{S;F}(n)$ is defined for all $n\in\N$. Similarly, we see that $S$ is multiplicatively piecewise $F$-syndetic if and only if the domain of $\tau_{S;F}$ is multiplicatively thick. 

This technique of identifying a multiplicatively syndetic set with its encoding function was the key idea in the proof of Proposition \ref{propSSPR}. A similar argument can be used to obtain the following analogous result.

\begin{proposition}[Piecewise syndetic sets contain PR configurations]\label{propPSSPR}
	Let $S\subseteq\N$ be a multiplicatively piecewise $F$-syndetic set, for some non-empty finite set $F\subset\N$. Let $\spt$ be a dilation invariant system of equations, and let $r\in\N$. If $\spt$ is $(|F|\cdot r)$-regular over $\N$, then $\spt$ is $r$-regular over $S$.
\end{proposition}
\begin{proof}
	Suppose $\chi:S\to[r]$ is an $r$-colouring of $S$. Since $S$ is multiplicatively piecewise syndetic, the set $T:=\cup_{t\in F}\,(t^{-1}\cdot S)$ is multiplicatively thick. Let $\tau:T\to F$ denote the encoding function given by
	\begin{equation*}
	\tau(n):=\min\{t\in F:nt\in S\}.
	\end{equation*}
	Now let $\tilde{\chi}:T\to F\times[r]$ be the product colouring given by
	\begin{equation*}
	\tilde{\chi}(n):=(\tau(n),\chi(n\tau(n))).
	\end{equation*}
	By Proposition \ref{propThickEqv}, we know that $\spt$ is $(|F|\cdot r)$-regular over $T$. Thus, we can find a $\tilde{\chi}$-monochromatic solution $\vx$ to $\spt$ such that every entry of $\vx$ is an element of $T$. From the definition of $\tilde{\chi}$, we can find $t\in F$ such that $\tau(x_{i})=t$ for every entry $x_{i}$ of $\vx$. The dilation invariance of $\spt$ then shows that $\vy:=t\vx$ is a $\chi$-monochromatic solution to $\spt$ whose entries $y_{i}=tx_{i}$ all lie in $S$.
\end{proof}

We end this section by synthesising all of our results relating partition regularity with solubility in multiplicatively syndetic sets into the following theorem.

\begin{theorem}[Summary of results]
	Suppose $\spt$ is a dilation invariant system of equations. Then the following are all equivalent:
	\begin{enumerate}[\upshape(I)]
		\item $\spt$ is partition regular (over $\N$);
		\item $\spt$ is partition regular over every multiplicatively thick set;
		\item $\spt$ is partition regular over every multiplicatively piecewise syndetic set;
		\item $\spt$ is partition regular over every multiplicatively syndetic set;
		\item $\spt$ has a solution in every multiplicatively piecewise syndetic set;
		\item $\spt$ has a solution in every multiplicatively syndetic set.
	\end{enumerate}
\end{theorem}
\begin{proof}
	Proposition \ref{propThickEqv} establishes the equivalence (I)$\Leftrightarrow$(II). Proposition \ref{propSSPR} and Theorem \ref{thmHomPR} together show that (I)$\Leftrightarrow$(IV)$\Leftrightarrow$(VI). Similarly, we deduce from Proposition \ref{propPSSPR} that (I)$\Leftrightarrow$(III). Since syndeticity implies piecewise syndeticity, we see that (V)$\Rightarrow$(VI). Finally, since solubility is equivalent to $1$-regularity, we observe that (III)$\Rightarrow$(V).
\end{proof}

\section{Density of Multiplicatively Syndetic Sets}\label{secDens}

In Ramsey theory, there are multiple concepts of `largeness'. The most familiar of these is the notion of (asymptotic) density.

\begin{definition}[Asymptotic Density]
	Let $A\subseteq\N$. The \emph{upper (asymptotic) density $\udn(A)$} of $A$ is defined by
	\begin{equation*}
	\udn(A):=\limsup_{N\to\infty}\frac{|A\cap[N]|}{N}.
	\end{equation*}
	The \emph{lower (asymptotic) density $\ldn(A)$} of $A$ is defined by
	\begin{equation*}
	\ldn(A):=\liminf_{N\to\infty}\frac{|A\cap[N]|}{N}.
	\end{equation*}
	The \emph{natural (asymptotic) density $d(A)$} of $A$ is defined by
	\begin{equation*}
	d(A):=\lim_{N\to\infty}\frac{|A\cap[N]|}{N},
	\end{equation*}
	whenever the above limit exists, which occurs if and only if $\udn(A)=\ldn(A)$.
\end{definition}
One can generalise this definition to obtain a notion of asymptotic density for general cancellative, left amenable semigroups (see \cite{BG16extra} for further details). The above definition comes from the case where the semigroup in question is $(\N,+)$. As such, asymptotic density is a form of `additive largeness'.

Recent research has led to the surprising discovery that multiplicatively large sets, such as multiplicatively syndetic sets, are additively large. Bergelson \cite[Lemma 5.11]{BerCent} proved that multiplicatively syndetic sets are additively central, which implies that they have positive upper asymptotic density. However, due to the infinitary nature of central sets, no explicit bounds on the density of multiplicatively syndetic sets can be extracted from this result.

In their work on the partition regularity of non-linear equations, Chow, Lindqvist, and Prendiville \cite[Lemma 4.2]{CLP} independently proved that multiplicatively syndetic sets have positive (lower) asymptotic density. They obtained the following quantitative result (for the case $F=[M]$).

\begin{lemma}\label{lemSynDens}
	Let $F\subset\N$ be a non-empty finite set, and let $M$ denote the largest element of $F$. Then for any $N\in\N$ and any multiplicatively $F$-syndetic set $S\subseteq[N]$, we have
	\begin{equation}
	|S\cap[N]|\geqslant \frac{1}{|F|}\left\lfloor\frac{N}{M}\right\rfloor.
	\end{equation}
\end{lemma}
\begin{proof}
	Define an encoding function $\tau:[N/M]\to F$ for $S$ by
	\begin{equation*}
	\tau(x):=\min\{t\in F:tx\in S \}.
	\end{equation*}
	By the pigeonhole principle, there exists $t\in F$ such that $|\tau^{-1}(t)|\geqslant \frac{1}{|F|}|[N/M]|$. Thus,
	\begin{equation*}
	|S|\geqslant |\{tx:x\in\tau^{-1}(t)\}|\geqslant\frac{1}{|F|}\left\lfloor\frac{N}{M}\right\rfloor.
	\end{equation*}
\end{proof}

In fact, the density of multiplicatively syndetic sets had been studied much earlier. In 1977 Graham, Spencer, and Witsenhausen \cite{GrahamSyndy} determined the maximum asymptotic density for sets lacking linear configurations of the form $\{a_{1}x,a_{2}x,\dots,a_{s}x\}$. Taking complements enables one to determine the minimum density of a multiplicatively $F$-syndetic set for $F=\{a_{1},\dots,a_{s}\}$. After performing this reformulation, their result is as follows.

\begin{theorem}[{\cite[Theorem 2]{GrahamSyndy}}]
	Let $F\subset\N$ be a non-empty finite set. Let $P_{F}$ be the set of primes dividing elements of $F$. Let $\mathrm{S}(P_{F})$ denote the set of all \emph{$P_{F}$-smooth numbers}, meaning that $x\in\mathrm{S}(P_{F})$ if and only if every prime factor of $x$ lies in $P_{F}$. We write $\mathrm{S}(P_{F})=\{d_{1}<d_{2}<d_{3}<\dots\}$, where $d_{k}$ is the $k$th smallest element of $\mathrm{S}(P_{F})$. For each $k\in\N$, let
	\begin{equation*}
	g_{F}(k):=\min\{|X\cap \{d_{1},\dots,d_{k}\}|:X\subseteq\N\;\text{is multiplicatively $F$-syndetic}\}
	\end{equation*}
	Finally, let $\mathrm{K}(F)=\{k\in\N:g_{F}(k)=g_{F}(k-1)\}$. Then for any multiplicatively $F$-syndetic set $S\subseteq\N$, we have the sharp bound
	\begin{equation*}
	\ldn(S)\geqslant\delta_{\text{min}}(F):= 1-\prod_{p\in\mathrm{S}(P_{F})}(1-p^{-1})\sum_{k\in \mathrm{K}(F)}d_{k}^{-1}.
	\end{equation*}
\end{theorem}

Graham, Spencer, and Witsenhausen remark that there are difficulties in evaluating $\delta_{\text{min}}(F)$ due to the complicated nature of the set $\mathrm{K}(F)$. In particular, they could not obtain an explicit evaluation for $\delta_{\text{min}}(F)$ in the case where $F=\{1,2,3\}$. Erd\H{o}s and Graham \cite{ErdosGrahamBra} subsequently conjectured that $\delta_{\text{min}}(\{1,2,3\})$ is irrational. This conjecture remains open (see \cite{ChungGrahamLin} for further details and developments related to this problem).

In the case where $F=\{1,p,p^{2},\dots,p^{k-1}\}$ for some prime $p$ and some $k\in\N$, Graham, Spencer, and Witsenhausen observed that $\delta_{\text{min}}(F)=\tfrac{p+1}{p^{k}-1}$. We now consider $F=\{1,a,a^{2},\dots,a^{k-1}\}$, where $a\in\N$ need not be prime, and explicitly construct a multiplicatively $F$-syndetic set of minimum density.

\begin{definition}[Multiplicity function]
	Let $a\geqslant 2$ be a positive integer. Define the \emph{$a$-multiplicity function} $\nu_{a}:\N\to\N\cup\{0\}$ by
	\begin{equation*}
		\nu_{a}(n)=\max\{k\in\N\cup\{0\}:a^{k}|n\}.
	\end{equation*}
\end{definition}

\begin{lemma}[The set $S(a,k)$]
	Let $a,k\in\N\setminus\{1\}$, and let $F=\{1,a,a^{2},\dots,a^{k-1}\}$. Let $S(a,k)\subseteq\N$ be defined by
	\begin{equation}\label{eqnSakDef}
	S(a,k):=\{n\in\N:\nu_{a}(n)\equiv k-1 \;(\bmod\; k)\}.
	\end{equation}
	Then $S(a,k)$ is a multiplicatively $F$-syndetic set and has natural density
	\begin{equation*}
	d\left(S(a,k)\right)=\frac{a-1}{a^{k}-1}.
	\end{equation*}
\end{lemma}
\begin{proof}
	By noting that $\nu_{a}(an)=\nu_{a}(n)+1$, we see that, for any $n\in\N$, the set $\nu_{a}(n\cdot F)$ is a complete residue system modulo $k$. Thus, $S(a,k)$ is a multiplicatively $F$-syndetic set. 
	
	It only remains to check that $S(a,k)$ achieves the required density bound.
	For each $m\in\N$, let $A_{m}=\{n\in\N:\nu_{a}(n)=km-1 \}$. Hence,
	\begin{equation*}
		S(a,k)=\bigcup_{m\in\N}A_{m}.
	\end{equation*}
	Observe that $n\in A_{m}$ holds if and only if $n\equiv a^{km-1}b\;(\bmod\; a^{km})$ for some $b\in\{1,2,\dots,a-1\}$. We therefore deduce that $A_{m}$ has natural density
	\begin{equation*}
	d(A_{m})=\frac{a-1}{a^{km}}.
	\end{equation*} 
	For each $r\in\N$, let $B_{r}=\cup_{m=1}^{r}A_{m}$.
	Since $A_{i}$ and $A_{j}$ are disjoint for all $i\neq j$, we deduce from the finite additivity of natural density that
	\begin{equation*}
		d(B_{r})=\sum_{m=1}^{r}d(A_{m})=(a-1)\sum_{m=1}^{r}a^{-km}=(a-1)\frac{1-a^{-rk}}{a^{k}-1}.
	\end{equation*}
	By noting that $B_{r}\subseteq S(a,k)$, we deduce that $\ldn\left( S(a,k)\right)\geqslant d(B_{r})$ for all $r\in\N$. Taking $r\to\infty$ gives the lower bound $\ldn(S(a,k))\geqslant(a-1)/(a^{k}-1)$.
	
	We now compute an upper bound. First observe that
	\begin{equation*}
	S(a,k)\setminus B_{r}\subseteq a^{kr-1}\cdot\N
	\end{equation*}
	for all $r\in\N$. Thus,
	\begin{equation*}
		\udn(S(a,k))\leqslant \udn(B_{r})+\udn(a^{kr-1}\cdot\N)=d(B_{r})+\frac{1}{a^{kr-1}}
	\end{equation*}
	holds for all $r\in\N$. Taking $r\to\infty$ gives the desired upper bound.
\end{proof}

We now show that $S(a,k)$ has minimal density.

\begin{theorem}[Minimal $\{1,a,\dots,a^{k-1}\}$-syndetic set]
	Let $a,k\in\N\setminus\{1\}$, and let $F=\{1,a,a^{2},\dots,a^{k-1}\}$. Let $S(a,k)$ be the set defined in (\ref{eqnSakDef}). Let $N\in\N$. If $X\subseteq\N$ is a multiplicatively $F$-syndetic set, then
	\begin{equation*}
	|X\cap[N]|\geqslant|S(a,k)\cap [N]|.
	\end{equation*}
\end{theorem}
\begin{proof}
	We may assume that $N\geqslant a^{k-1}$, since otherwise $S(a,k)\cap [N]=\emptyset$ and the result is vacuously true. Let $X\subseteq\N$ be a multiplicatively $F$-syndetic set. Let $m\in S(a,k)\cap[N]$. Since every element of $S(a,k)$ is divisible by $a^{k-1}$, we deduce that $(a^{-(k-1)}m)\cdot F\subseteq[N]$. As $X$ is multiplicatively $F$-syndetic, we can find some $t(m)\in F$ such that $a^{-(k-1)}mt(m)\in X$. We can therefore define a function $g:S(a,k)\cap[N]\to X$ by $g(n)=a^{-(k-1)}nt(n)$.
	To complete the proof it is sufficient to show that $g$ is an injective function. 
	
	Suppose that $n,n'\in S(a,k)$ satisfy $g(n)=g(n')$. Thus, $nt(n)=n't(n')$. Applying $\nu_{a}$ to this equation and then reducing modulo $k$ gives
	\begin{equation*}
	\nu_{a}(t(n))\equiv\nu_{a}(t(n'))\;(\bmod\; k).
	\end{equation*}
	Now observe that the function which maps $t\in F$ to the residue class of $\nu_{a}(t)$ modulo $k$ is injective. This shows that $t(n)=t(n')$, which implies that $n=n'$.
\end{proof}

\section{A Syndetic Density Increment Strategy}\label{secBrau}

Brauer \cite{Bra28} established the following common generalisation of Schur's theorem and van der Waerden's theorem.\\

\noindent\textbf{Brauer's Theorem:} \emph{For all $k,r\in\N$, there exists $N_{0}=N_{0}(k,r)\in\N$ such that, in any $r$-colouring $[N_{0}]=C_{1}\cup\cdots\cup C_{r}$ of $[N_{0}]=\{1,2,\dots,N_{0}\}$, there is a colour class $C_{i}$ containing a set of the form}
\begin{equation}
\{x,d,x+d,x+2d,\dots,x+(k-1)d\}.
\end{equation}

\noindent In this section we study the case where $k=3$. This corresponds to configurations of the form
\begin{equation}
\{x,d,x+d,x+2d\}.\label{eqnBrauer}
\end{equation}
We refer to such configurations as \emph{Brauer configurations (of length $3$)}. These are three term arithmetic progressions along with their common difference. Alternatively, one can view Brauer configurations as being solutions $\{x,y,z,d\}$ to the following dilation invariant system of equations:
\begin{align*}
x-2y+z=0;\\
x-y+d=0.
\end{align*}

As shown by Theorem \ref{thmHomPR}, Brauer's theorem is equivalent to the assertion that all multiplicatively syndetic sets contain Brauer configurations. In this section, we derive quantitative bounds on the minimal $N\in\N$ for which the set $S\cap[N]$ must contain a Brauer configuration, for a given multiplicatively syndetic set $S$. The main theorem of this section is as follows.

\begin{theorem}[Syndetic Brauer]\label{thmSynDensInc}
	Let $M\geqslant 2$ be a positive integer. Let $A\subseteq[N]$ be such that $|A|\geqslant\delta N$ for some $0<\delta\leqslant 1/2$. Let $S\subseteq\N$ be a multiplicatively $[M]$-syndetic set. If there does not exist a $3$-term arithmetic progression in $A$ with common difference in $S$, then
	\begin{equation}\label{eqnMainNbd}
	\log\log N\ll \log(\delta^{-1})\log(M/\delta).
	\end{equation}
\end{theorem}

\begin{remark}
	We impose the restriction $\delta\leqslant 1/2$ to ensure that $\delta$ is bounded away from $1$. One could replace $1/2$ by any positive quantity strictly less than $1$ at the cost of increasing the implicit constant in (\ref{eqnMainNbd}).
\end{remark}

In \cite{GT09}, Green and Tao used a density increment strategy over quadratic factors to obtain new bounds for the sizes of subsets of $[N]$ lacking $4$-term arithmetic progressions. In this section we combine their methods with the induction on colours argument (Lemma \ref{lemColInd}) of Chow, Lindqvist, and Prendiville \cite{CLP} to prove Theorem \ref{thmSynDensInc}.

By incorporating the density bounds obtained in \S\ref{secDens}, we can use Theorem \ref{thmSynDensInc} to prove Theorem \ref{thmSynBrau}.

\begin{proof}[Proof of Theorem \ref{thmSynBrau} given Theorem \ref{thmSynDensInc}]
	Let $c>0$ be a small positive constant to be specified later, and assume that $M,N\in\N$ satisfy
	\begin{equation*}
	2\leqslant M\leqslant\exp\left(c\sqrt{\log\log N} \right). 
	\end{equation*}
	Let $S\subseteq\N$ be a multiplicatively $F$-syndetic set, for some non-empty $F\subseteq[M]$. By taking $c$ sufficiently small, we may assume that $N\geqslant M$. This implies that the density of $S\cap[N]$ in $[N]$ is positive. Moreover, by Lemma \ref{lemSynDens}, we can choose a subset $S'\subseteq S$ such that the density $\delta$ of $S'\cap[N]$ in $[N]$ satisfies both of the bounds $0<\delta\leqslant 1/2$ and $\delta^{-1}\ll M^{2}$. This gives
	\begin{equation*}
	\log(\delta^{-1})\log(M/\delta)\ll \log^{2} M.
	\end{equation*}
	Hence, by choosing $c$ to be sufficiently small, we can ensure that (\ref{eqnMainNbd}) does not hold. Thus, by taking $A=S'\cap[N]$, we conclude from Theorem \ref{thmSynDensInc} that $S\cap[N]$ contains a Brauer configuration.
\end{proof}

For each $r\in\N$, define the \emph{$r$ colour Brauer number} $\rB(r)$ to be the minimum $N\in\N$ such that any $r$-colouring of $[N]$ yields a monochromatic Brauer configuration. In other words, $\rB(r)$ is the minimal value that $N_{0}(3,r)$ may take in the statement of Brauer's theorem. We can use Theorem \ref{thmSynDensInc} to obtain the following recursive bound for these numbers.

\begin{theorem}[Recursive bound for $\rB(r)$]\label{thmMain}
	For each $r\in\N$, we have
	\begin{equation}\label{eqnMainRecbd}
	\rB(r+1)\leqslant 2^{\rB(r)^{O(\log(r+1))}}.
	\end{equation}
\end{theorem}

\noindent By some computation (which is given in Appendix \ref{secTowerBd}), this theorem leads to the tower type bound in Theorem \ref{thmBrauerBound}.

\begin{proof}[Proof of Theorem \ref{thmMain} given Theorem \ref{thmSynDensInc}]
	Let $r\in\N$. Let $M:=\rB(r)$, and let $\delta:=(r+1)^{-1}$. Note that $M\geqslant\rB(1)=3$. Suppose $N\in\N$ is such that $N<\rB(r+1)$. Therefore, we have an $(r+1)$-colouring $[N]=C_{1}\cup\cdots\cup C_{r+1}$ with no monochromatic sets of the form (\ref{eqnBrauer}). By dilation invariance and our choice of $M$, it follows that there cannot exist a set of the form $x\cdot[M]\subseteq[N]$ which is $r$-coloured. This implies that $C_{i}\cup(\N\setminus[N])$ is multiplicatively $[M]$-syndetic for all $i\in[r]$.
	
	Without loss of generality, assume that $C_{1}$ is the largest colour class. By the pigeonhole principle, we observe that $|C_{1}|\geqslant \delta N$. Hence, by taking $A=C_{1}$ and $S=C_{1}\cup(\N\setminus[N])$ in the statement of Theorem \ref{thmSynDensInc}, we deduce that
	\begin{equation*}
	\log\log \rB(r+1)\ll \log(r+1)\log\left((r+1)\rB(r)\right).
	\end{equation*}
	By noting that $\rB(r)\geqslant r+1$ (since one can $r$-colour $[r]$ so that each element has a unique colour), this gives
	\begin{equation*}
	\log\log \rB(r+1)\ll \log(r+1)\log(\rB(r)).
	\end{equation*}
	Exponentiating twice then gives (\ref{eqnMainRecbd}).
\end{proof}

\subsection{Norms}

Given a non-empty finite set $A$ and a function $f:A\to\C$, we define the expectation $\E_{A}(f)$ of $f$ over $A$ by
\begin{equation*}
\mathbb{E}_{A}(f)=\E_{x\in A}(f):=\frac{1}{|A|}\sum_{x\in A}f(x).
\end{equation*}
The functions we encounter in this section are usually defined on $[N]$ or $\Z/p\Z$, where $N\in\N$, and $p\in\N$ is a prime. When $p>N$, it is convenient to consider $[N]$ as a subset of the field $\Z/p\Z$ by reducing modulo $p$. Given a function $f:\Z\to\C$ which is supported on $[N]$, we can then consider $f$ as a function defined on $\Z/p\Z$ by taking $f(x)=0$ for all $x\in(\Z/p\Z)\setminus[N]$.

We make use of two different types of norms. The standard $L^{p}$ norms are used to measure the overall size of a function, whilst the Gowers uniformity $U^{s}$ norms (introduced in \cite[Lemma 3.9]{Gow01}) measure the degree to which a function exhibits non-uniformity.

\begin{definition}[$L^{p}$ norms]
	Let $A$ be a set and let $f:A\to\C$ be a finitely supported function.
	The \emph{$L^{1}$ norm} $\lVert f\rVert_{L^{1}(A)}$ of $f$ is defined by
	\begin{equation*}
	\lVert f\rVert_{L^{1}(A)}:=\E_{x\in A}|f(x)|.
	\end{equation*}
	The \emph{$L^{\infty}$ norm} $\lVert f\rVert_{L^{\infty}(A)}$ of $f$ is defined by
	\begin{equation*}
	\lVert f\rVert_{L^{\infty}(A)}:=\max_{x\in A}|f(x)|.
	\end{equation*}
	We say that $f$ is \emph{$1$-bounded} (on $A$) if $\lVert f\rVert_{L^{\infty}(A)}\leqslant 1$.\\
\end{definition}

\begin{definition}[$U^{s}$ norms]
	Let $f:\Z/p\Z\to\C$. For each $s\geqslant 1$, the \emph{$U^{s}$ norm}\footnote{For $s\geqslant 2$ the $U^{s}$ norms are indeed norms (see \cite[Chapter 11]{TV06}), however the $U^{1}$ `norm' is only a seminorm.} $\lVert f\rVert_{U^{s}(\Z/p\Z)}$ of $f$ is defined by
	\begin{equation}\label{eqnFinGowNorm}
	\lVert f\rVert_{U^{s}(\Z/p\Z)} :=\left(\E_{x,h_{1},h_{2},\dots,h_{s}\in\Z/p\Z}\,\Delta_{h_{1},\dots,h_{s}}f(x)\right) ^{1/2^{s}},
	\end{equation}
	where the \emph{difference operators} $\Delta_{h_{1},\dots,h_{s}}$ are defined by
	\begin{equation*}
	\Delta_{h}f(x):=f(x)\overline{f(x+h)}
	\end{equation*}
	and
	\begin{equation*}
	\Delta_{h_{1},\dots,h_{s}}f:=\Delta_{h_{1}}\Delta_{h_{2}}\cdots\Delta_{h_{s}}f.
	\end{equation*}
\end{definition}
\noindent The Gowers uniformity norms can also be defined recursively. By expanding and rearranging (\ref{eqnFinGowNorm}), we observe that
\begin{equation}\label{eqnRecGowDef}
\lVert f\rVert_{U^{s+1}}^{2^{s+1}}=\E_{h\in\Z/p\Z}\lVert\Delta_{h}f\rVert_{U^{s}}^{2^{s}}.
\end{equation}

\subsection{Counting Brauer Configurations}

To prove Theorem \ref{thmMain}, we use an induction on colours argument similar to \cite[\S4]{CLP}.

For the remainder of this section, we let $N$ denote a positive integer and consider Brauer configurations in the interval $[N]$. It is useful to embed $[N]$ in an abelian group which is not much larger than $[N]$. We therefore let $p$ denote a prime\footnote{\,Such a prime exists by Bertrand's postulate.} satisfying $3N<p<6N$, and embed $[N]$ in the group $\Z/p\Z$ by reducing modulo $p$. We also let $M\in\N$ be a positive integer with $M>1$ so that we may consider multiplicatively $[M]$-syndetic sets.

Let $S\subseteq\Z/p\Z$ and let $f_{1},f_{2},f_{3}:\Z/p\Z\to\C$. The counting functional $\Lambda_{S}$ we use is defined by
\begin{equation*}
\Lambda_{S}(f_{1},f_{2},f_{3}):=\BigE_{x\in\Z/p\Z}\BigE_{d\in S}f_{1}(x)f_{2}(x+d)f_{3}(x+2d).
\end{equation*}
For brevity, we write $\Lambda_{S}(f):=\Lambda_{S}(f,f,f)$.

\begin{lemma}[Counting Brauer configurations with common difference in $S$]\label{lemBrauCount}$ $
	\\
	Let $N\in\N$ with $N\geqslant 3$. If $S\subseteq[N/3]$ is non-empty, then
	\begin{equation*}
	\Lambda_{S}(1_{[N]})>\frac{1}{18}.
	\end{equation*}
\end{lemma}
\begin{proof}
	Since $S\subseteq[N/3]$, for all $d\in S$ we have $N-2d\geqslant N/3$. Combining this with the bound $p<6N$ gives
	\begin{equation*}
	\Lambda_{S}(1_{[N]})=\frac{1}{p}\BigE_{d\in S}(N-2d)\geqslant\frac{N}{3p}>\frac{1}{18}.
	\end{equation*}
\end{proof}

We now use the two different types of norms introduced earlier to control the size of $\Lambda_{S}$. The simplest way to bound $\Lambda_{S}$ is by using the $L^{1}$ norm.

\begin{lemma}[$L^{1}$ control for $\Lambda_{S}$]\label{lemL1DiffCon}
	Let $S\subseteq[N]$ and let $f,g:\Z/p\Z\to\C$ be $1$-bounded functions. Then we have
	\begin{equation}\label{eqnL1ctrl}
	|\Lambda_{S}(f)-\Lambda_{S}(g)|\leqslant 3\lVert f-g\rVert_{L^{1}(\Z/p\Z)}.
	\end{equation}
\end{lemma}
\begin{proof}
	First let $k\in\{1,2,3\}$ and let $f_{1},f_{2},f_{3}:\Z/p\Z\to\C$ be functions such that $f_{i}$ is $1$-bounded for all $i\neq k$. By a change of variables, we see that
	\begin{align*}
	|\Lambda_{S}(f_{1},f_{2},f_{3})|&\leqslant\BigE_{d\in S}\BigE_{x\in\Z/p\Z}|f_{1}(x)f_{2}(x+d)f_{3}(x+2d)|\\
	&\leqslant \BigE_{d\in S}\BigE_{x\in\Z/p\Z}|f_{k}(x+(k-1)d)1_{S}(d)|\\
	&=\left(\BigE_{y\in\Z/p\Z}|f_{k}(y)|\right)\left(\BigE_{d\in S}1_{S}(d) \right).  
	\end{align*}
	We therefore deduce that
	\begin{equation}\label{eqnL1Cont}
	|\Lambda_{S}(f_{1},f_{2},f_{3})|\leqslant \lVert f_{k}\rVert_{L^{1}(\Z/p\Z)}
	\end{equation}
	holds for all $k\in\{1,2,3\}$.
	
	Now observe that, by multilinearity, we have the telescoping identity
	\begin{equation}\label{eqnTeleId}
	\Lambda_{S}(f)-\Lambda_{S}(g)=\Lambda_{S}(f-g,f,f)+\Lambda_{S}(g,f-g,f)+\Lambda_{S}(g,g,f-g).
	\end{equation}
	Applying the triangle inequality to this identity and using (\ref{eqnL1Cont}) gives (\ref{eqnL1ctrl}).
\end{proof}

In addition to $\Lambda_{S}$, for functions $f_{1},f_{2},f_{3}:\Z/p\Z\to\C$, we introduce the auxiliary counting functional $AP_{3}$  given by
\begin{equation*}
AP_{3}(f_{1},f_{2},f_{3}):=\BigE_{x,d\in\Z/p\Z}f_{1}(x)f_{2}(x+d)f_{3}(x+2d).
\end{equation*}
Since Brauer configurations contain three term arithmetic progressions, it is perhaps unsurprising that the uniformity of $\Lambda_{S}$ is related to the uniformity of $AP_{3}$. Indeed, the original motivation for the introduction of the $U^{s}$ norms in \cite[\S3]{Gow01} was the observation that they control counting functionals for arithmetic progressions. This result is referred to in the literature as a \emph{generalised von Neumann theorem}. In the case of three term arithmetic progressions, the result is as follows.

\begin{lemma}[Generalised von Neumann theorem]\label{lemAP3bd}
	Let $f_{1},f_{2},f_{3}:\Z/p\Z\to\C$ be $1$-bounded functions. Then we have
	\begin{equation}\label{eqnAP3bd}
	|AP_{3}(f_{1},f_{2},f_{3})|\leqslant \min_{1\leqslant k\leqslant 3}\lVert f_{k}\rVert_{U^{2}(\Z/p\Z)}.
	\end{equation}
\end{lemma}
\begin{proof}
	This follows from two applications of the Cauchy-Schwarz inequality. For the full details, see \cite[Lemma 11.4]{TV06}.
\end{proof}

We now prove an analogous result for the $\Lambda_{S}$ functional.

\begin{lemma}[Generalised von Neumann theorem for $\Lambda_{S}$]\label{lemLamControl}
	Let $S\subseteq[N]$ be a non-empty set, and let $f_{1},f_{2},f_{3}:\Z/p\Z\to\C$ be $1$-bounded functions. Then
	\begin{equation*}
	|\Lambda_{S}(f_{1},f_{2},f_{3})|\leqslant \frac{p^{1/2}}{|S|^{1/2}}\min_{1\leqslant k\leqslant 3}\lVert f_{k}\rVert_{U^{3}(\Z/p\Z)}.
	\end{equation*}
\end{lemma}
\begin{proof}
	Observe that we can rewrite $\Lambda_{S}(f_{1},f_{2},f_{3})$ as
	\begin{equation*}
	\Lambda_{S}(f_{1},f_{2},f_{3})=\frac{p}{|S|}\BigE_{x,d\in\Z/p\Z}f_{1}(x)f_{2}(x+d)f_{3}(x+2d)1_{S}(d).
	\end{equation*}
	By applying the Cauchy-Schwarz inequality (with respect to the $d$ variable), we see that the quantity $|\Lambda_{S}(f_{1},f_{2},f_{3})|^{2}$ is bounded above by
	\begin{align*}
	&\left( \BigE_{d\in\Z/p\Z}\frac{p^{2}}{|S|^{2}}1_{S}(d)\right) \left(\BigE_{d\in\Z/p\Z}\left\lvert\BigE_{x,d\in\Z/p\Z}f_{1}(x)f_{2}(x+d)f_{3}(x+2d) \right\rvert^{2} \right)\\
	&=\frac{p}{|S|} \BigE_{d,x,x'\in\Z/p\Z}f_{1}(x)f_{1}(x')f_{2}(x+d)f_{2}(x'+d)f_{3}(x+2d)f_{3}(x'+2d)\\
	&=\frac{p}{|S|} \BigE_{d,h,x\in\Z/p\Z}\Delta_{h}f_{1}(x)\Delta_{h}f_{2}(x+d)\Delta_{h}f_{3}(x+2d)\\
	&=\frac{p}{|S|}\BigE_{h\in\Z/p\Z}AP_{3}(\Delta_{h}f_{1},\Delta_{h}f_{2},\Delta_{h}f_{3}).
	\end{align*}
	Note that the penultimate equality above follows from a changed of variables of the form $h=x'-x$. Now let $k\in\{1,2,3\}$. Using Lemma \ref{lemAP3bd} and (\ref{eqnRecGowDef}) along with an application of H\"{o}lder's inequality gives
	\begin{align*}
	|\Lambda_{S}(f_{1},f_{2},f_{3})|^{2} 
	&\leqslant\frac{p}{|S|}\BigE_{h\in\Z/p\Z}\lVert \Delta_{h}f_{k}\rVert_{U^{2}(\Z/p\Z)}\\
	&\leqslant\frac{p}{|S|}\left(\BigE_{h\in\Z/p\Z}1^{4/3}\right)^{3/4}\left(\BigE_{h\in\Z/p\Z}\lVert \Delta_{h}f_{k}\rVert_{U^{2}(\Z/p\Z)}^{4} \right)^{1/4} \\
	&=\frac{p}{|S|}\lVert f_{k}\rVert_{U^{3}(\Z/p\Z)}^{2}.
	\end{align*}
\end{proof}
\begin{lemma}[$U^{3}$ controls $\Lambda_{S}$]\label{lemU3DiffCon}
	Let $S\subseteq\N$ be a multiplicatively $[M]$-syndetic set, and let $f,g:\Z/p\Z\to[0,1]$. If $N\geqslant 18M^{2}$, then
	\begin{equation}\label{eqnU3ctrl}
	|\Lambda_{S\cap[N/3]}(f)-\Lambda_{S\cap[N/3]}(g)|\leqslant 18M\lVert f-g\rVert_{U^{3}(\Z/p\Z)}.
	\end{equation}
\end{lemma}
\begin{proof}
	Note that as $f$ and $g$ are non-negative, the difference $f-g$ is a $1$-bounded function. Applying the previous lemma and the triangle inequality to the telescoping identity (\ref{eqnTeleId}) then gives
	\begin{equation*}
	|\Lambda_{S\cap[N/3]}(f)-\Lambda_{S\cap[N/3]}(g)|\leqslant \frac{3p^{1/2}}{|S\cap[N/3]|^{1/2}}\lVert f-g\rVert_{U^{3}(\Z/p\Z)}.
	\end{equation*}
	Combining Lemma \ref{lemSynDens} with the assumption $N\geqslant 18M^{2}$ gives
	\begin{equation*}
	|S\cap[N/3]|\geqslant\frac{N}{6M^{2}}>\frac{p}{36M^{2}}.
	\end{equation*}
	This implies the desired bound (\ref{eqnU3ctrl}).
\end{proof}

We now proceed to prove Theorem \ref{thmSynDensInc} using a density increment strategy. This argument combines Green and Tao's quadratic Fourier analytic methods \cite{GT09} for finding sets lacking arithmetic progressions of length $4$ with the techniques used by Chow, Lindqvist, and Prendiville \cite[Lemma 7.1]{CLP} to obtain a `homogeneous' generalisation of S\'{a}rk\"{o}zy's theorem \cite{Sarkoz}.

The original density increment strategy of Roth \cite{Roth} was used to show that subsets of $[N]$ which lack arithmetic progressions of length $3$ have size $o(N)$. This method was subsequently modified by Gowers to prove an analogous result for arithmetic progressions of length 4 \cite{Gow 4AP}, and then further generalised for progressions of arbitrary length \cite{Gow01}. The argument proceeds as follows. Let $\delta_{0}>0$. Suppose $A\subseteq[N]$ lacks arithmetic progressions of length $3$ and satisfies $|A|=\alpha N$ for some $\alpha\geqslant\delta_{0}$. Then provided that $N$ is `not too small', meaning that $N>C(\delta_{0})$ for some positive constant $C(\delta_{0})$ depending only on $\delta_{0}$, we can find an arithmetic progression $P\subseteq[N]$ of length $N':=|P|\geqslant F(N,\delta_{0})$ on which $A$ has a density increment
\begin{equation}\label{eqnMethDen}
\alpha':=\frac{|A\cap P|}{|P|}\geqslant \alpha + c(\delta_{0}).
\end{equation}
Here $c(\delta_{0})>0$ is a positive constant depending only on $\delta_{0}$, and $F$ is an explicit positive function such that, for any fixed $\delta>0$, $F(N,\delta)\to\infty$ as $N\to\infty$. 

We can then apply an affine transformation of the form $x\mapsto ax+b$ to injectively map $A\cap P$ into $[N']$ with image $A'\subseteq[N']$. Since arithmetic progressions are translation-dilation invariant, we deduce that $A'$ also lacks arithmetic progressions of length $3$ and satisfies $|A'|=\alpha'N'>\delta_{0} N'$. We can then iterate this argument. Since the density is increasing by at least $c(\delta_{0})$ after each iteration, this process must eventually terminate. We can then procure an upper bound for the size of the original $N$ in terms of $C(\delta_{0}),c(\delta_{0})$ and $F(\cdot\, ,\delta_{0})$.

An important aspect of this method is that it uses the translation-dilation invariance of arithmetic progressions. However, more general configurations, such as Brauer configurations, are not translation invariant. This is emphasised by the fact that the odd numbers have density $1/2$ and yet they do not contain any Brauer configurations. A density analogue of Brauer's theorem is therefore impossible, and so this argument cannot help us prove Brauer's theorem.

The major insight of Chow, Lindqvist, and Prendiville \cite[\S7]{CLP} is that one can separate such configurations into a `translation invariant part' and a `non-translation invariant part'. For instance, observe that if $\{x,d,x+d,x+2d\}$ is a Brauer configuration, then the set $\{x+h,d,(x+d)+h,(x+2d)+h\}$ is also a Brauer configuration for any $h\in\N$. Brauer configurations therefore consist of a translation invariant part $\{x,x+d,x+2d\}$ and a non-translation invariant part $\{d\}$.

This allows us to modify the density increment strategy of Roth and Gowers to prove Brauer's theorem. Instead of studying a single set $A$ lacking Brauer configurations, we study a pair of sets $A$ and $S$ with the following properties.
\begin{enumerate}[(i)]
	\item (Density). $A\subseteq[N]$ satisfies $|A|\geqslant \delta N$.
	\item (Syndeticity). $S\subseteq\N$ is a multiplicatively $[M]$-syndetic set.
	\item (Brauer free). There does not exist an arithmetic progression of length $3$ in $A$ with common difference in $S\cap[N/3]$.
\end{enumerate}

As in the original density increment argument, we show that, provided $N$ is `not too small', we can find a long arithmetic progression $P\subseteq[N]$ on which we have a density increment of the form (\ref{eqnMethDen}). As before, we can apply an affine transformation to obtain a new set $A'\subseteq[N']$ with increased density. We can also obtain a new multiplicatively $[M]$-syndetic set $S'=d^{-1}S$, where $d$ is the common difference of the progression $P$.

Recall that the translation invariant part $\{x,x+d,x+2d\}$ of a Brauer configuration is required to come from the dense set $A$, whilst the non-translation invariant part $\{d\}$ comes from the multiplicatively syndetic set $S$. Thus we have obtained new sets $A',S'\subseteq[N']$ satisfying (i)-(iii). Iterating this procedure as in the Gowers-Roth argument allows us to prove Theorem \ref{thmSynDensInc}.

\begin{theorem}[Density increment for Brauer configurations]\label{thmDensInc}
	There exists a constant $C_{0}>1$ such that the following is true.
	Let $A\subseteq[N]$ be such that $|A|\geqslant\delta N$, for some $\delta>0$, and let $S\subseteq\N$ be a multiplicatively $[M]$-syndetic set. Suppose that there do not exist $x\in A$ and $d\in S\cap[N/3]$ such that $\{x,x+d,x+2d\}\subseteq A$. If $N$ satisfies
	\begin{equation}\label{eqnCondLrgN}
	N> \exp(C_{0}\delta^{-C_{0}}M^{C_{0}}),
	\end{equation}
	then there exists positive constants $C,c>0$ and an arithmetic progression $P$ in $[N]$ satisfying 
	\begin{equation*}
	|P|\gg N^{c\delta^{C}M^{-C}}
	\end{equation*}
	such that we have the density increment
	\begin{equation*}
	\frac{|A\cap P|}{|P|}\geqslant\left( 1+c\right)\delta.
	\end{equation*}
\end{theorem}

\begin{remark}
	The bound (\ref{eqnCondLrgN}) is needed to ensure that $N$ is not too small to satisfy the conclusion of the theorem. Moreover, by taking $C_{0}\geqslant 2$, we can assume that $N\geqslant 18M^{2}$. This allows us to make use of Lemma \ref{lemU3DiffCon}. 
\end{remark}

\begin{proof}[Proof of Theorem \ref{thmSynDensInc} given Theorem \ref{thmDensInc}]
	Let $C_{1}\geqslant C_{0}$ be a large positive parameter (which does not depend on $M$ or $\delta$) to be specified later. We use the following iteration algorithm. After the $i\,$th iteration, we have a positive integer $N_{i}\in\N$, a positive real number $\delta_{i}\geqslant\delta >0$, an infinite set $S_{i}\subseteq\N$, and a finite set $A_{i}\subseteq[N_{i}]$ satisfying the following three properties: 
	\begin{enumerate}[(I)]
		\item $|A_{i}|\geqslant\delta_{i}|N_{i}|$;
		\item $S_{i}$ is a multiplicatively $[M]$-syndetic set;
		\item there does not exist a $3$-term arithmetic progression in $A_{i}$ with common difference in $S_{i}\cap[N_{i}/3]$.
	\end{enumerate}
	We begin by defining the initial variables $
	N_{0}:=N,\,
	A_{0}:=A,\,
	S_{0}:=S,\,
	\delta_{0}:=\delta.$
	The iteration step of the algorithm proceeds as follows. If after the $i\,$th iteration we have
	\begin{equation}\label{eqnPfLargeCon}
	N_{i}\leqslant \exp(C_{1}\delta^{-C_{1}}M^{C_{1}}),
	\end{equation}
	then the algorithm terminates. If not, then we can apply Theorem \ref{thmDensInc} with $N=N_{i}$ and $f=1_{A_{i}}$ to obtain an arithmetic progression $P_{i}\subseteq[N_{i}]$ of the form
	\begin{equation*}
	P_{i}=\{a_{i},a_{i}+d_{i},\dots,a_{i}+(|P_{i}|-1)d_{i}\}
	\end{equation*}
	which satisfies the length bound
	\begin{equation}\label{eqnProgLen}
	|P_{i}|\gg N_{i}^{c\delta^{C}M^{-C}},
	\end{equation}
	and provides the density increment
	\begin{equation}\label{eqnPfDenInc}
	\frac{|A_{i}\cap P_{i}|}{|P_{i}|}\geqslant\left(1+c\right)\delta_{i}.
	\end{equation}
	Moreover, by partitioning $P_{i}$ into two shorter progressions if necessary, we can ensure that $d_{i}|P_{i}|\leqslant N_{i}$, provided that $C_{1}$ is sufficiently large. We then take
	\begin{align*}
	N_{i+1}&:=|P_{i}|;\\
	A_{i+1}&:=\{x\in[N_{i+1}]:a_{i}+(x-1)d_{i}\in A_{i}\cap P_{i}\};\\
	S_{i+1}&:=d_{i}^{-1}S_{i}=\{x\in\N: d_{i}x\in S_{i}\};\\
	\delta_{i+1}&:=\frac{|A_{i+1}|}{|N_{i+1}|}.
	\end{align*}
	We now claim that $(N_{i+1},A_{i+1},S_{i+1},\delta_{i+1})$ satisfy properties (I), (II), and (III). Property (I) follows immediately from our choice of $\delta_{i+1}$. Property (II) follows from the fact that $x\cdot [M]$ intersects $d^{-1}_{i}S_{i}$ if and only if $(d_{i}x)\cdot[M]$ intersects $S_{i}$. Finally, notice that if $A_{i+1}$ contains a $3$-term arithmetic progression with common difference $q$, then $A_{i}$ contains a $3$-term arithmetic progression with common difference $d_{i}q$. Since $d_{i}|P_{i}|\leqslant N_{i}$, we conclude that $(N_{i+1},A_{i+1},S_{i+1},\delta_{i+1})$ satisfies property (III).
	
	We have therefore shown that our algorithm may continue with the new variables $(N_{i+1},A_{i+1},S_{i+1},\delta_{i+1})$. Moreover, after applying the iteration process $t$ times, we see from (\ref{eqnPfDenInc}) that the density $\delta_{t}$ satisfies $\delta_{t}\geqslant (1+c)^{t}\delta$. Since $\delta_{t}\leqslant 1$ for all $t$, we conclude that the algorithm must terminate after $T$ steps for some $T\ll \log(\delta^{-1})$.
	
	We therefore deduce that (\ref{eqnPfLargeCon}) must hold for $i=T$. By (\ref{eqnProgLen}), we see that
	\begin{equation*}
	N_{T}\geqslant N^{O(c^{T}\delta^{CT}M^{-CT})}.
	\end{equation*}
	Combining these two bounds for $N_{T}$ gives
	\begin{equation*}
	c^{T}\delta^{CT}M^{-CT}\log N\ll C_{1}\delta^{-C_{1}}M^{C_{1}}.
	\end{equation*}
	Rearranging and taking logarithms gives
	\begin{equation*}
	\log\log N\ll \log C_{1} +T\log(c^{-1})+(C_{1}+CT)\log\left(M/\delta\right)\ll \log(\delta^{-1})\log\left(M/\delta\right).
	\end{equation*}
	We have therefore established (\ref{eqnMainNbd}), as required.
\end{proof}

\subsection{Quadratic Fourier Analysis for Brauer Configurations}

The goal of the rest of this section is to prove Theorem \ref{thmDensInc}. We achieve this by adapting the methods used by Green and Tao \cite{GT09} to study subsets of $[N]$ which lack arithmetic progressions of length 4. 

The objective of their argument is to show that if a subset $A\subseteq[N]$ of density $\alpha$ lacks arithmetic progressions of length $4$, then there exists a long arithmetic progression $P\subseteq[N]$ upon which $A$ achieves a density increment $|A\cap P|\geqslant(\alpha+c(\alpha))|P|$. Gowers' argument yields an increment of the form $c(\alpha)\gg\alpha^{C}$. The key insight of Green and Tao is that one can obtain a larger density increment if one first shows that $A$ has a density increment on a `quadratic Bohr set'. A linearisation procedure can then be applied to obtain a long arithmetic progression $P'$ which provides a density increment $c(\alpha)\gg \alpha$.

\begin{remark}
	In this subsection we make use of several results from \cite{GT09}. The statements of these theorems contain a number of technical terms from quadratic Fourier analysis. We assume the reader is familiar with the notation used in \cite[\S\S3-5]{GT09} and employ this henceforth without further explanation.
\end{remark}

The first theorem we need is \cite[Theorem 5.6]{GT09}.

\begin{theorem}[Quadratic Koopman-von Neumann theorem]\label{BB1}
	Let $\eps>0$ and let $f:\Z/p\Z\to[-1,1]$. Suppose $K\in\N$ satisfies $K\geqslant C\eps^{-C}$ for some absolute constant $C>0$. Then there exists a quadratic factor $(\cB_{1},\cB_{2})$ in $\Z/p\Z$ of complexity at most $(O(\eps^{-C}),O(\eps^{-C}))$ and resolution $K$ such that
	\begin{equation}\label{eqnQuadKoopBd}
	\lVert f-\E(f|\cB_{2}\vee\cB_{triv})\rVert_{U^{3}(\Z/p\Z)}\leqslant\eps.
	\end{equation}
\end{theorem}

This theorem allows us to approximate (in the $U^{3}$ norm) a $1$-bounded function $f$ with a more `structured' function $g:=\E(f|\cB_{2}\vee\cB_{triv})$. In the proof of Theorem \ref{thmDensInc} we take $f=1_{A}$, where $A\subseteq[N]$ is a dense subset of $[N]$ which lacks arithmetic progressions of length $3$ with common difference in $S\cap[N/3]$, for a given multiplicatively syndetic set $S$. Our goal is to obtain a density increment on a quadratic factor for this $f$. To do this, we first show that it is sufficient to obtain a density increment with respect to the approximation $g$.

\begin{corollary}[Brauer configurations on a quadratic factor]\label{cor5.7}$ $
	\newline
	Let $f:\Z/p\Z\to[0,1]$ be a $1$-bounded non-negative function which is supported on $[N]$. Let $\delta>0$. Suppose that $N\geqslant 18M^{2}$ and
	\begin{equation}\label{eqnCondLam}
	|\Lambda_{S\cap[N/3]}(f)-\Lambda_{S\cap[N/3]}(\delta 1_{[N]})|\geqslant\delta^{3}/18.
	\end{equation}
	Then there exists a quadratic factor $(\cB_{1},\cB_{2})$ in $\Z/p\Z$ of complexity at most $(O(\delta^{-C}M^{C}),O(\delta^{-C}M^{C}))$ and resolution $O(\delta^{-C}M^{C})$ such that
	\begin{equation}\label{eqnManyBra}
	|\Lambda_{S\cap[N/3]}(g)-\Lambda_{S\cap[N/3]}(\delta 1_{[N]})|\geqslant\delta^{3}/36,
	\end{equation}
	where $g:=\E(f|\cB_{2}\vee\cB_{triv})$.
\end{corollary}
\begin{proof}
	Let $\eps=(\delta^{3}M^{-1})/648$. By Theorem \ref{BB1}, for some absolute constant $C>0$, we have a quadratic factor $(\cB_{1},\cB_{2})$ in $\Z/p\Z$ of complexity at most $(O(\delta^{-C}M^{C}),O(\delta^{-C}M^{C}))$ and resolution $O(\delta^{-C}M^{C})$ such that (\ref{eqnQuadKoopBd}) holds.

	From Lemma \ref{lemU3DiffCon} and (\ref{eqnQuadKoopBd}) we deduce
	\begin{equation*}
	|\Lambda_{S\cap[N/3]}(f)-\Lambda_{S\cap[N/3]}(g)|\leqslant 18M\eps=\delta^{3}/36.
	\end{equation*}
	An application of the triangle inequality to (\ref{eqnCondLam}) then gives (\ref{eqnManyBra}).
\end{proof}

We now follow the approach of Green and Tao \cite[Corollary 5.8]{GT09} by replacing $f$ with $\E(f|\cB_{2}\vee\cB_{triv})$ to obtain a density increment on a quadratic factor.

\begin{corollary}[Density increment on a quadratic Bohr set]\label{cor5.8}
	There exists a constant $\tilde{C}_{0}\geqslant 2$ such that the following is true. Let $S\subseteq\N$ be a multiplicatively $[M]$-syndetic set, and let $f:\Z/p\Z\to[0,1]$ be supported on $[N]$. Suppose $\E_{[N]}(f)\geqslant\delta$, for some $\delta>0$. Suppose further that conditions (\ref{eqnCondLrgN}) and (\ref{eqnCondLam}) both hold, for some $C_{0}\geqslant \tilde{C}_{0}$. Then there exists a quadratic factor $(\cB_{1},\cB_{2})$ in $\Z/p\Z$ of complexity at most $(O(\delta^{-C}M^{C}),O(\delta^{-C}M^{C}))$ and resolution $O(\delta^{-C}M^{C})$, and an atom $B$ of the factor $\cB_{2}\vee\cB_{triv}$ with density $\frac{|B|}{p}\gg\exp(-O(\delta^{-C}M^{C}))$ which is contained in $[N]$ and is such that
	\begin{equation}\label{eqnBexp}
	\mathbb{E}_{B}(f)\geqslant\left( 1+c\right) \delta,
	\end{equation}
	for some absolute constant $c>0$.
\end{corollary}
\begin{proof}
	Let $(\cB_{1},\cB_{2})$ be the quadratic factor obtained from Corollary \ref{cor5.7}. Let $g=\E(f|\cB_{2}\vee\cB_{triv})$.  Note that $[N]\in\cB_{triv}$, and so $[N]\in\cB_{2}\vee\cB_{triv}$. Since $g$ is constant on atoms of $\cB_{2}\vee\cB_{triv}$ and $f$ is supported on $[N]$, we see that $g$ is also supported on $[N]$. This implies that $\mathbb{E}_{B}(f)=\mathbb{E}_{B}(g)$ holds for any $B\in\cB_{2}\vee\cB_{triv}$. Thus, it is sufficient to establish (\ref{eqnBexp}) with $g$ in place of $f$.
	
	Let $\eta>0$ be a small constant (to be chosen later) and define 
	\begin{equation*}
		\Omega:=\left\lbrace x\in[N]:g(x)\geqslant\left(1+\eta\right)\delta\right\rbrace.
	\end{equation*}
	Since $g$ is constant on atoms of $\cB_{2}\vee\cB_{triv}$, we deduce that $\Omega$ can be partitioned into atoms of $\cB_{2}\vee\cB_{triv}$. We can therefore finish the proof if we can show that one of these atoms $B$ satisfies
	\begin{equation}\label{eqnPigAtom}
		|B|\gg \exp(-O(\delta^{-C}M^{C}))p.
	\end{equation}
	Now recall that the factor $\cB_{2}\vee\cB_{triv}$ has complexity and resolution $O(\delta^{-C}M^{C})$, and so contains at most $\exp(O(\delta^{-C}M^{C}))$ atoms. Thus the pigeonhole principle implies that (\ref{eqnPigAtom}) holds if we can obtain a bound of the form
	\begin{equation}\label{eqnMuOmegaBd}
		|\Omega|\gg\delta^{3}p.
	\end{equation}
	Define the function $h:\Z/p\Z\to\C$ by
	\begin{equation*}
		h(x)=1_{[N]\setminus\Omega}(x)g(x).
	\end{equation*}
	Thus $\lVert h\rVert_{L^{\infty}(\Z/p\Z)}<\left(1+\eta\right)\delta$. Taking $\eta\leqslant 1$ and applying Lemma \ref{lemL1DiffCon} gives
	\begin{equation*}
		|\Lambda_{S\cap[N/3]}(h)-\Lambda_{S\cap[N/3]}(\delta 1_{[N]})|\leqslant 12\delta^{2}\lVert h-\delta 1_{[N]} \rVert_{L^{1}(\Z/p\Z)}.
	\end{equation*}
	From the fact that $g$ is a $1$-bounded function, we have
	\begin{equation*}
		\lVert h-\delta 1_{[N]}\rVert_{L^{1}(\Z/p\Z)}\leqslant \lVert g-\delta 1_{[N]}\rVert_{L^{1}(\Z/p\Z)}+\frac{|\Omega|}{p}.
	\end{equation*}
	Since $h$ is a $1$-bounded function, Lemma \ref{lemL1DiffCon} also gives
	\begin{equation*}
		|\Lambda_{S\cap[N/3]}(g)-\Lambda_{S\cap[N/3]}(h)|\leqslant \frac{3|\Omega|}{p}.
	\end{equation*}
	Combining these three bounds and using the triangle inequality in (\ref{eqnManyBra}) gives
	\begin{equation}\label{eqnPfDeltaMu}
		\delta^{2}\lVert g-\delta 1_{[N]}\rVert_{L^{1}(\Z/p\Z)}+\frac{|\Omega|}{p} \gg\delta^{3}.
	\end{equation}
	Recall that $\E_{[N]}(g)=\E_{[N]}(f)\geqslant\delta$. Thus,
	\begin{align*}
		\lVert g-\delta 1_{[N]}\rVert_{L^{1}(\Z/p\Z)}&\leqslant \lVert g-\delta 1_{[N]}\rVert_{L^{1}(\Z/p\Z)}+\E_{\Z/p\Z}(g-\delta 1_{[N]})\\
		&=2\lVert (g-\delta 1_{[N]})_{+}\rVert_{L^{1}(\Z/p\Z)}\\
		&\ll \frac{|\Omega|}{p}+\delta\eta.
	\end{align*}
	If $\eta$ is sufficiently small (relative to the implicit constants), we can then substitute this bound into (\ref{eqnPfDeltaMu}) to obtain the desired result (\ref{eqnMuOmegaBd}).
\end{proof}

To complete the proof of Theorem \ref{thmDensInc} it only remains to convert this density increment on a quadratic Bohr set into a density increment on an arithmetic progression. This is accomplished by implementing the following `linearisation' procedure of Green and Tao \cite[Proposition 6.2]{GT09}.

\begin{theorem}[Linearisation of quadratic Bohr sets]\label{BB2}
	Suppose $(\cB_{1},\cB_{2})$ is a quadratic factor in $\Z/p\Z$ of complexity at most $(d_{1},d_{2})$ and resolution $K$, for some $K\in\N$. Let $B_{2}$ be an atom of $\cB_{2}$. Then for all $N\in\N$, there is a partition of $B_{2}\cap[N]$ as a union of $\ll d_{2}^{\,O(d_{2})}N^{1-c/(d_{1}+1)(d_{2}+1)^{3}}$ disjoint arithmetic progressions in $\Z/p\Z$.
\end{theorem}

\begin{proof}[Proof of Theorem \ref{thmDensInc}]
	Let $\tilde{C}_{0}\geqslant 2$ be the positive constant appearing in the statement of Corollary \ref{cor5.8}. Let $A$ and $S$ be as defined in the statement of Theorem \ref{thmDensInc}. Suppose $N$ satisfies (\ref{eqnCondLrgN}), for some $C_{0}\geqslant\tilde{C}_{0}$. Let $f=1_{A}$. By Corollary \ref{cor5.8} there is a quadratic factor $(\cB_{1},\cB_{2})$ in $\Z/p\Z$ of complexity at most $(O(\delta^{-C}M^{C}),O(\delta^{-C}M^{C}))$ and resolution $O(\delta^{-C}M^{C})$, and an atom $B\subseteq[N]$ of $\cB_{2}\vee\cB_{triv}$ with density $\frac{|B|}{p}\geqslant\exp(-O(\delta^{-C}M^{C}))$ such that $E_{B}(f)\geqslant\left(1+c_{0}\right)\delta$. By Theorem \ref{BB2}, we may write $B$ as a union of $\exp(O(\delta^{-C}M^{C}))N^{1-c\delta^{C}M^{-C}}$ arithmetic progressions. By an application of the pigeonhole principle (see \cite[Lemma 6.1]{GT09}), we deduce that one of these progressions $P$ satisfies
	\begin{equation*}
	|P|\geqslant \exp(-O(\delta^{-C}M^{C}))N^{c\delta^{C}M^{-C}}
	\end{equation*}
	and
	\begin{equation*}
	\frac{|A\cap P|}{|P|}=\mathbb{E}_{P}f\geqslant \left(1+\frac{c_{0}}{2} \right)\delta.
	\end{equation*}
	Notice that the only property of the parameter $C_{0}$ appearing in (\ref{eqnCondLrgN}) that we have used is that $C_{0}\geqslant\tilde{C}_{0}$. We may therefore take $C_{0}$ to be sufficiently large so that $|P|\gg N^{c'\delta^{C'}M^{-C'}}$ holds for some absolute constants $C',c'>0$. This completes the proof.
\end{proof}

\appendix

\section{Obtaining a Tower Bound}\label{secTowerBd}

In this section, we show how Theorem \ref{thmMain} implies Theorem \ref{thmBrauerBound}. By Theorem \ref{thmMain}, there is a positive constant $\tilde{C}>0$ such that 
\begin{equation*}
\rB(r+1)\leqslant 2^{\rB(r)^{\tilde{C}\log(r+1)}}
\end{equation*}
holds for all $r\in\N$.

Given $n\in\N$ and $a_{1},\dots,a_{n}\in[2,\infty)$, define the tower function
\begin{equation*}
T_{n}(a_{1},a_{2},\dots,a_{n}):=a_{1}^{a_{2}^{\iddots^{a_{n}}}}.
\end{equation*}
Let $K\geqslant 1$ be a large positive constant, to be chosen later. We can now introduce the auxiliary function $F:\N\to\R$ given by
\begin{equation*}
F(r):=T_{r+1}(2,2,\dots,2,Kr^{2}).
\end{equation*}
Thus, we have
\begin{equation*}
F(1)=2^{K},\hspace{1em}F(2)=2^{2^{4K}},\hspace{1em}F(3)=2^{2^{2^{9K}}},\hspace{1em}F(4)=2^{2^{2^{2^{16K}}}},\hspace{1em}\dots.
\end{equation*}

Our goal is to show that, if $K$ is sufficiently large relative to $\tilde{C}$, then
\begin{equation*}
\rB(r)\leqslant F(r)
\end{equation*}
holds for all $r\in\N$. To demonstrate why this is enough to prove Theorem \ref{thmBrauerBound}, we first need to investigate the growth of tower functions.
\begin{lemma}[Towers dominate cubes]
	For all $r\geqslant 5$ we have
	\begin{equation}\label{eqnCubeTowbd}
	r^{3}\leqslant\Tow(r-1).
	\end{equation}
\end{lemma}
\begin{proof}
	We first observe that (\ref{eqnCubeTowbd}) holds for $r=5$. Suppose then that $r>5$ and assume the induction hypothesis that
	\begin{equation*}
	(r-1)^{3}\leqslant\Tow(r-2).
	\end{equation*}
	Note that since $r>5$, we have
	\begin{equation*}
	\frac{r^{3}}{(r-1)^{3}}=\left(1+\frac{1}{r-1} \right)^{3}<\frac{125}{64}<2.
	\end{equation*}
	This gives
	\begin{equation*}
	r^{3}\leqslant 2(r-1)^{3}\leqslant 2\cdot\Tow(r-2).
	\end{equation*}
	Using the elementary fact that $2k\leqslant 2^{k}$ holds for all $k\in\N$, we deduce
	\begin{equation*}
	r^{3}\leqslant 2^{\Tow(r-2)}=\Tow(r-1).
	\end{equation*}
	The desired result now follows by induction.
\end{proof}

This lemma enables us to bound $F$ above by a tower function.

\begin{corollary}[Tower bound for $F$]\label{corTowF}
	For all $r\in\N$,
	\begin{equation}\label{eqnApfBd}
	F(r)\leqslant\Tow\left((1+o(1))r\right).
	\end{equation}
\end{corollary}
\begin{proof}
	Recall that $F$ is an exponential tower of height $r+1$, with $K r^{2}$ as the `top' term, and with the remaining terms in the tower equal to $2$. By the previous lemma, when $r$ is sufficiently large, we have
	\begin{equation*}
	K r^{2}\leqslant\Tow\left(\lceil K^{1/3}r^{2/3}\rceil \right). 
	\end{equation*}
	By adding the heights of the towers, we deduce that
	\begin{equation*}
	F(r)\leqslant\Tow\left(r+\left\lceil K^{1/3}r^{2/3}\right\rceil\right)
	\end{equation*}
	holds for all sufficiently large $r$. This gives (\ref{eqnApfBd}).
\end{proof}

We require the following elementary result concerning manipulations of exponentials.

\begin{lemma}\label{lemLift}
	For all $a,b,k\geqslant 2$,
	\begin{equation}
	a^{b}k\leqslant a^{b+k}\leqslant a^{bk}\label{eqnLift}.
	\end{equation}
\end{lemma}

This gives the following bound on the growth of $F$.
\begin{corollary}[Tower growth of $F$]\label{corTowGrowF}
	For all $r\in\N$,
	\begin{equation*}
	F(r)^{r}\leqslant\log_{2}F(r+1).
	\end{equation*}
\end{corollary}
\begin{proof}
	Since $K\geqslant 1$, the case $r=1$ can be verified by inspection. Suppose then that $r\geqslant 2$. By iteratively applying Lemma \ref{lemLift}, we deduce that
	\begin{align*}
	\log_{2}F(r+1)&=T_{r+1}(2,2,\dots,2,2,K(r+1)^{2})\\
	&\geqslant T_{r}(2,2,\dots,2,2^{K r^{2}}\cdot r)\\
	&\geqslant T_{r-1}(2,2,\dots,2^{2^{K r^{2}}}\cdot r)\\
	& \vdots\\
	&\geqslant T_{2}(2,T_{r}(2,2,\dots,2,K r^{2})\cdot r).\\
	&=F(r)^{r}.
	\end{align*}
\end{proof}

We can now prove Theorem \ref{thmBrauerBound}.

\begin{proof}[Proof of Theorem \ref{thmBrauerBound}]
	By Corollary \ref{corTowF}, it is sufficient to show that
	\begin{equation}\label{eqnBrauF}
	\rB(r)\leqslant F(r)
	\end{equation}
	holds for all $r\in\N$. Since $\tilde{C}\log(n+1)=o(n)$, we can choose $n_{0}\in\N$ with $n_{0}\geqslant 5$ such that
	\begin{equation*}
	\tilde{C}\log(n+1)\leqslant n
	\end{equation*}
	holds for $n\geqslant n_{0}$. By taking $K$ sufficiently large, we can assume that (\ref{eqnBrauF}) holds for $r\leqslant n_{0}$. Suppose then that $r>n_{0}$ and assume the induction hypothesis
	\begin{equation*}
	\rB(r-1)\leqslant F(r-1).
	\end{equation*}
	By Theorem \ref{thmMain} and the fact that $r>n_{0}$, we have
	\begin{align*}
	\log_{2}\rB(r)&\leqslant \rB(r-1)^{\tilde{C}\log r }\leqslant \rB(r-1)^{r-1}.
	\end{align*}
	Now by the induction hypothesis and Corollary \ref{corTowGrowF}, we conclude that
	\begin{equation*}
	\log_{2}\rB(r)\leqslant F(r-1)^{r-1}\leqslant \log_{2}F(r).
	\end{equation*}
	This establishes the induction step and completes the proof.
\end{proof}

\end{document}